\documentclass[a4paper, 12pt]{amsart}
\usepackage{hyperref,upref,amssymb,tikz}

\numberwithin{equation}{section}

\newtheorem{thm}{Theorem}[section]
\newtheorem{lemma}[thm]{Lemma}
\newtheorem{prop}[thm]{Proposition}
\newtheorem{cor}[thm]{Corollary}
\theoremstyle{definition}
\newtheorem{defn}[thm]{Definition}

\theoremstyle{remark}
\newtheorem{eg}[thm]{Example}
\newtheorem{rmk}[thm]{Remark}

\newcommand{\id}{\operatorname{id}}

\newcommand{\dom}{\mathrm{dom}}
\newcommand{\ran}{\mathrm{ran}}

\title{Purely Infinite Totally Disconnected Topological Graph Algebras}

\date{24 Feb 2017}

\author{Hui Li}
\email{lihui8605@hotmail.com}

\address{Research Center for Operator Algebras and Shanghai Key Laboratory of Pure Mathematics and Mathematical Practice, Department of Mathematics, East China Normal University, 3663 Zhongshan North Road, Putuo District, Shanghai 200062, China
}

\subjclass[2010]{46L05}
\keywords{$C^*$-algebra; topological graph; Renault-Deaconu groupoid; groupoid $C^*$-algebra; pure infiniteness}
\thanks{This research was supported by Research Center for Operator Algebras of East China Normal University and supported by Science and Technology Commission of Shanghai Municipality (STCSM), grant No. 13dz2260400.}

\begin{document}

\begin{abstract}
We give a sufficient condition on totally disconnected topological graphs such that their associated topological graph algebras are purely infinite.
\end{abstract}

\maketitle

\section{Introduction}

The Elliott program developed rapidly over the last twenty years with the goal of classifying all simple separable nuclear $C^*$-algebras by the so-called Elliott's invariants (Notice that different types of classification might have different invariants). In the purely infinite case, Kirchberg in \cite{Kirchbergclassification} and Phillips in \cite{Phillips:DM00} separately showed that all simple separable nuclear purely infinite $C^*$-algebras in the UCT class can be classified by their K-theoretic data. Katsura in \cite{Katsura:JFA08} gave a sufficient condition on simple topological graph algebras such that they are purely infinite, and constructed all simple separable nuclear purely infinite $C^*$-algebras in the UCT class.

Graph algebras were firstly defined by Kumjian, Pask, Raeburn, and Renault in \cite{KumjianPaskEtAl:JFA97} using the groupoid $C^*$-algebras method. Topological graph algebras studied by Katsura (see \cite{Katsura:TAMS04}), which are seen to be a generalization of graph algebras, were defined by using a modified version of Pimsner's construction (see \cite{Pimsner:FIC97}). A few people then tried to realize topological graph algebras as groupoid $C^*$-algebras, and these results can naturally be regarded as a generalization of Kumjian, Pask, Raeburn, and Renault's approach to graph algebras in \cite{KumjianPaskEtAl:JFA97}. For example, Katsura in \cite{MR2563503} showed that the topological graph algebra of a compact topological graph with a surjective range map is isomorphic to a Renault-Deaconu groupoid $C^*$-algebra. Yeend in \cite{Yeend:CM06} proved that every topological graph algebra is indeed an \'{e}tale groupoid $C^*$-algebra. Kumjian and Li in \cite{KumjianLiTwisted2} strengthened Yeend's result by showing that every topological graph algebra is indeed a Renault-Deaconu groupoid $C^*$-algebra (Yeend's result, and Kumjian and Li's result both cover a result of Brownlowe, Carlsen, and Whittaker from \cite[Proposition~2.2]{BCW}).

In this article, we give a sufficient condition on totally disconnected topological graphs such that their associated topological graph algebras are purely infinite. Our approach is different than Katsura's. Our strategy is that we deal with topological graph algebras under the groupoid model, and we apply Anantharaman-Delaroche's criterion in \cite{MR1478030}, which yields purely infinite \'{e}tale groupoid $C^*$-algebras, to our settings.

This paper is organized as follows. In Section~2, we give a background review on topological graph algebras and state the main theorem of \cite{KumjianLiTwisted2}. In Section~3, we prove our main theorem, that is Theorem~\ref{sufficient condition 1 for boundary path groupoid loc con}, which is a sufficient condition giving rise to purely infinite topological graph algebras. In Section~4, we give some remarks on our main theorem.

\section{Preliminaries}

Throughout this paper, all the topological spaces are assumed to be second countable; and all the topological groupoids are assumed to be second countable. Our work in this article utilizes the Hilbert module, the $C^*$-correspondence, and the Cuntz-Pimsner algebra machinery. These materials can be referred to \cite{Katsura:JFA04, Lance:Hilbert$C*$-modules95, Pimsner:FIC97, RaeburnWilliams:Moritaequivalenceand98}, etc. This paper also involves groupoids, groupoid $C^*$-algebras, which can be found in \cite{Renault:groupoidapproachto80}.

\subsection{Topological Graph Algebras}

In this subsection, we recap some background about topological graphs and topological graph algebras from \cite{Katsura:TAMS04, KumjianLiTwisted2}.

\begin{defn}
Let $T$ be a locally compact Hausdorff space. Then $T$ is said to be \emph{totally disconnected} if $T$ has an open base consisting of compact open subsets of $T$.
\end{defn}

\begin{defn}[{\cite[Definition~2.1]{Katsura:TAMS04}}]\label{define the top graph}
A quadruple $E=(E^0,E^1,r,s)$ is called a \emph{topological graph} if $E^0, E^1$ are locally compact Hausdorff spaces, $r:E^1 \to E^0$ is a continuous map, and $s:E^1\to E^0$ is a local homeomorphism. In addition, $E$ is said to be \emph{totally disconnected} if $E^0, E^1$ are both totally disconnected.
\end{defn}

\begin{defn}[{\cite{Katsura:TAMS04}}]
Let $E$ be a topological graph. For $x,y \in C_c(E^1), f \in C_0(E^0), e \in E^1$, and for $v \in E^0$, define
\[
(x \cdot f)(e):=x(e)f(s(e));\ (f \cdot x)(e):=f(r(e))x(e); \text{ and } \langle x,y \rangle_{C_0(E^0)}(v):=\sum_{s(e)=v}\overline{x(e)}y(e).
\]
Then $C_c(E^1)$ is a right inner product $C_0(E^0)$-module with an adjointable left $C_0(E^0)$-action. Its completion $X(E)$ under the $\Vert\cdot\Vert_{C_0(E^0)}$-norm is called the \emph{graph correspondence} associated to $E$. The Cuntz-Pimsner algebra of $X(E)$, which is denote by $\mathcal{O}(E)$, is called the \emph{topological graph algebra} of $E$.
\end{defn}

A subset $N$ of $E^1$ is called an \emph{$s$-section} if $s\vert_{N}:N \to s(N)$ is a homeomorphism with respect to the subspace topologies.

Define some useful subsets of $E^0$ as follows. Define
\begin{enumerate}
\item $E_{\mathrm{sce}}^0:=E^0 \setminus \overline{r(E^1)}$.
\item $E_{\mathrm{fin}}^0 := \{v \in E^0 :$ there exists an open neighborhood $N$ of $v$ such
    that $r^{-1}(\overline{N})$ is compact $\}$.
\item $E_{\mathrm{rg}}^0:=E_{\mathrm{fin}}^0 \setminus \overline{E_{\mathrm{sce}}^0}$.
\item $E_{\mathrm{sg}}^0:=E^0 \setminus E_{\mathrm{rg}}^0$.
\end{enumerate}

For $n \geq 2$, define
\[
E^n:=\Big\{\mu =(\mu_1,\dots,\mu_n)\in \prod_{i=1}^{n}E^1: s(\mu_i)=r(\mu_{i+1}), i=1,\dots,n-1 \Big\}
\]
regarded as a subspace of the product space $\prod_{i=1}^{n}E^1$. Define the \emph{finite-path space} $E^*:=\coprod_{n=0}^{\infty} E^n$ with the disjoint union topology. Define the \emph{infinite-path space}
\[
E^\infty:=\Big\{\mu\in \prod_{i=1}^{\infty}E^1: s(\mu_i)=r(\mu_{i+1}), i=1,2,\dots \Big\}.
\]
Denote the length of a path $\mu \in E^* \amalg E^\infty$ by $\vert\mu\vert$.

A finite path $\mu \in E^* \setminus E^0$ is called a \emph{cycle} if $r(\mu)=s(\mu)$. The vertex $r(\mu)$ is called the \emph{base point} of $\mu$. The cycle $\mu$ is said to be \emph{without entrances} if $r^{-1}(r(\mu_i))=\{\mu_i\}$, for $i=1, \dots, \vert\mu\vert$. On the other hand, the cycle $\mu$ is said to \emph{have entrances} if there exists $1 \leq i \leq \vert\mu\vert$ such that $r^{-1}(r(\mu_i))\neq\{\mu_i\}$.

\begin{defn}[{\cite[Definition~5.4]{Katsura:TAMS04}}]
Let $E$ be a topological graph. Then $E$ is said to be \emph{topologically free} if the set of base points of cycles without entrances has empty interior.
\end{defn}

\begin{defn}[{\cite[Definitions~4.1, 4.7]{KumjianLiTwisted2}}]\label{def of boundary path of top graph}
Let $E$ be a topological graph. Define the \emph{boundary path space} to be $\partial E:=E^\infty \amalg \{\mu\in E^*:s(\mu) \in E_{\mathrm{sg}}^0 \}$. For a subset $S \subset E^*$, define the \emph{cylinder set} by $Z(S):=\{\mu \in \partial E: \text{ there exists } \alpha \in S, \text{ such that } \mu=\alpha\beta\}$.
Define a locally compact Hausdorff topology on $\partial E$ to be generated by the basic open sets $Z(U) \cap Z(K)^c$, where $U$ is an open set of $E^*$ and $K$ is a compact set of $E^*$.
\end{defn}

By \cite[Lemma~7.1]{KumjianLiTwisted2}, the one-sided shift map $\sigma:\partial E \setminus E_{\mathrm{sg}}^0 \to \partial E$ is a local homeomorphism.

\begin{defn}
Let $E$ be a topological graph and let $v \in E^0$. Define the \emph{positive orbit space} of $v$ (see \cite[Definition~4.1]{Katsura:ETDS06}) by
\[
\mathrm{Orb}^+(v):=\{w \in E^0: \text{ there exists } \nu \in E^*, \text{ such that } r(\nu)=w,s(\nu)=v \}.
\]
$E$ is said to be \emph{cofinal} if for any $\mu \in \partial E, \mathrm{Orb}^+(r(\mu)) \bigcup \Big(\bigcup_{i=1}^{\vert\mu\vert} \mathrm{Orb}^+(s(\mu_i)) \Big)$ is dense in $E^0$.
\end{defn}

\begin{rmk}
The cofinality for topological graphs is a generalization of the cofinality for directed graphs (see \cite{KumjianPaskEtAl:PJM98, KumjianPaskEtAl:JFA97}). Let $E$ be a directed graph, the cofinality of $E$ means that for any $v \in E^0$, and for any $\mu \in \partial E$, there exists $\nu \in E^*$, such that $r(\nu)=v, s(\nu) \in \{r(\mu)\} \bigcup \Big(\bigcup_{i=1}^{\vert\mu\vert}\{s(\mu_i) \} \Big)$.
\end{rmk}

\begin{thm}[{\cite[Proposition~8.9, Theorem~8.12]{Katsura:ETDS06}}]\label{simplicity}
Let $E$ be a topological graph. Then $\mathcal{O}(E)$ is simple if and only if $E$ is topologically free and cofinal.
\end{thm}

\subsection{Groupoid $C^*$-algebras}

In this subsection, we state one of the main theorems of \cite{KumjianLiTwisted2}.

A topological groupoid is called a \emph{locally compact groupoid} if its topology is locally compact Hausdorff. A locally compact groupoid is said to be \emph{\'{e}tale} if its range map is a local homeomorphism.

Let $\Gamma$ be a locally compact groupoid and let $N \subset \Gamma$. Then $N$ is called an \emph{$s$-section} if $s\vert_{N}:N \to s(N)$ is a homeomorphism with respect to the subspace topologies; $N$ is called an \emph{$r$-section} if $r\vert_{N}:N \to r(N)$ is a homeomorphism with respect to the subspace topologies; and $N$ is called a \emph{bisection} if $s\vert_{N}, r\vert_{N}$ are both homeomorphisms with respect to the subspace topologies.

\begin{defn}[{\cite[Definition~2.4]{Renault:00}}]\label{define D-R groupoid}
Let $T$ be a locally compact Hausdorff space and let $\sigma:\dom(\sigma) \to \ran(\sigma)$ be a partial local homeomorphism. Define the \emph{Renault-Deaconu groupoid} $\Gamma(T,\sigma)$ as follows:
\begin{align*}
\Gamma(T,\sigma):=\{(t_1,k_1-k_2,t_2) \in T \times \mathbb{Z} \times T :& k_1,k_2 \geq 0,t_1 \in \dom(\sigma^{k_1}),
\\&t_2 \in \dom(\sigma^{k_2}), \sigma^{k_1}(t_1)=\sigma^{k_2}(t_2) \}.
\end{align*}
Define the unit space $\Gamma^0:=\{(t,0,t):t \in T\}$. For $(t_1,n,t_2),(t_2,m,t_3) \in \Gamma(T,\sigma)$, define the multiplication, the inverse, the source and the range map by
\[
(t_1,n,t_2)(t_2,m,t_3):=(t_1,n+m,t_3); (t_1,n,t_2)^{-1}:=(t_2,-n,t_1);
\]
\[
 r(t_1,n,t_2):=(t_1,0,t_1); s(t_1,n,t_2):=(t_2,0,t_2).
\]
Define the topology on $\Gamma(T,\sigma)$ to be generated by the basic open set
\[
\mathcal{U}(U,V,k_1,k_2):=\{(t_1,k_1-k_2,t_2):t_1 \in U,t_2 \in V,\sigma^{k_1}(t_1)=\sigma^{k_2}(t_2)\},
\]
where $U \subset \dom(\sigma^{k_1}), V \subset \dom(\sigma^{k_2})$ are open in $T,\sigma^{k_1}$ is injective on $U$, and $\sigma^{k_2}$ is injective on $V$.
\end{defn}

\begin{defn}[{\cite[Definition~7.5]{KumjianLiTwisted2}}]\label{boundary-path gpoid}
Let $E$ be a topological graph. Define the \emph{boundary path groupoid} to be the Renault-Deaconu groupoid
\[
\Gamma(\partial E,\sigma):=\{(\mu,k-l,\nu) \in \partial E \times \mathbb{Z} \times \partial E:\mu \in \dom(\sigma^{k}),\nu \in \dom(\sigma^l),\sigma^{k}(\mu)=\sigma^l(\nu)\}.
\]
\end{defn}

The following theorem is a special case of \cite[Theorem~7.6]{KumjianLiTwisted2}.

\begin{thm}\label{groupid model for twisted top graph alg}
Let $E$ be a topological graph. Then $\mathcal{O}(E)$ is isomorphic to the groupoid $C^*$-algebra $C^*(\Gamma(\partial E,\sigma))$.
\end{thm}

\section{Sufficient Conditions of Purely Infinite Topological Graph Algebras}

\begin{defn}[{\cite{RordamStormer:Classificationofnuclear02}}]\label{define purely infinite}
Let $A$ be a $C^*$-algebra. Then $A$ is said to be \emph{purely infinite} if every nonzero hereditary $C^*$-subalgebra of $A$ contains an infinite projection.
\end{defn}

We firstly prove the following two technical lemmas.

\begin{lemma}\label{Technical Lemma 1}
Let $E$ be a topological graph. Fix an open set $U \subset E^*$, and fix a compact set $K \subset E^*$ satisfying $Z(U) \cap Z(K)^c\neq\emptyset$. Write $K=\bigcup_{i=0}^{k}(K\cap E^i)$ for some $k \geq 0$. Then for any $\mu \in Z(U) \cap Z(K)^c$ with $\vert\mu\vert\geq k$, there exists an open subset $V$ of $E^*$ such that $\mu \in Z(V) \subset Z(U) \cap Z(K)^c$. In particular, if $\vert\mu\vert=k$, then $V$ can be chosen to be an open neighborhood of $\mu$ in $E^{\vert\mu\vert}$.
\end{lemma}
\begin{proof}
We prove the first statement. Write $\mu=\alpha\beta$ where $\alpha \in U$.

Case $1$. $K=\emptyset$. Let $V:=U$. Then we are done.

Case $2$. $K \neq\emptyset$ and $K \subset E^0$. Then $r(\mu) \notin K$. Take an open neighborhood $N$ of $r(\mu)$ which does not intersect with $K$. Let $V:=(r^{\vert\alpha\vert})^{-1}(N) \cap U$. We have $\mu\in Z(V) \subset Z(U) \cap Z(K)^c$.

Case $3$. $K\neq\emptyset$ and $K \not\subset E^0$. Then $k \geq 1$. Since $\mu \in Z(K)^c$, then for $1 \leq i \leq k$ there exists an open neighborhood $N_i \subset E^i$ of $\mu_1\dots\mu_i$ such that $N_i$ does not intersect with $K \cap E^i$ and $r(N_i)$ does not intersect with $K \cap E^0$. Let
\begin{align*}
V:= \begin{cases}
   \Big(\bigcap_{i=1}^{k-1}(N_i \times E^{k-i})\Big) \cap N_k \cap ((U \cap E^{\vert\alpha\vert})\times E^{k-\vert\alpha\vert}) &\text{ if $k>\vert\alpha\vert >0$} \\
        \Big(\bigcap_{i=1}^{k}(N_i \times E^{\vert\alpha\vert-i})\Big) \cap (U \cap E^{\vert\alpha\vert})   &\text{ if $\vert\alpha\vert>k$} \\
      \Big(\bigcap_{i=1}^{k-1}(N_i \times E^{k-i})\Big) \cap N_k \cap (U \cap E^{\vert\alpha\vert})   &\text{ if $\vert\alpha\vert =k$} \\
    \Big(\bigcap_{i=1}^{k-1}(N_i \times E^{k-i})\Big) \cap N_k \cap (r^k)^{-1}(U \cap E^{\vert\alpha\vert})   &\text{ if $\vert\alpha\vert =0$.}
\end{cases}
\end{align*}
Then $V$ is an open subset of $E^{\max\{\vert\alpha\vert,k\}}$ and $\mu\in Z(V) \subset Z(U) \cap Z(K)^c$.

We prove the second statement. If $K=\emptyset$, let
\begin{align*}
V:= \begin{cases}
   (U \cap E^{\vert\alpha\vert}) \times E^{\vert\beta\vert} &\text{ if $\vert\alpha\vert>0,\vert\beta\vert >0$} \\
   (r^{k})^{-1}(U) &\text{ if $\vert\alpha\vert=0,\vert\beta\vert >0$} \\
        U \cap E^{k} &\text{ if $\vert\beta\vert =0$.}
\end{cases}
\end{align*}
Then $V$ is an open neighborhood of $\mu$ in $E^{\vert\mu\vert}$, and $Z(V) \subset Z(U) \cap Z(K)^c$. If $K\neq\emptyset$, then it follows directly from the above construction.
\end{proof}

\begin{lemma}\label{Technical Lemma 2}
Let $E$ be a topological graph. Suppose that $E_{\mathrm{sce}}^0=\emptyset$. Then for any open set $U \subset E^*$ and any compact set $K \subset E^*$ satisfying $Z(U) \cap Z(K)^c\neq\emptyset$, there exists an open subset $V$ of $E^*$ such that $\emptyset\neq Z(V) \subset Z(U) \cap Z(K)^c$.
\end{lemma}
\begin{proof}
The assumption $E_{\mathrm{sce}}^0=\emptyset$ gives $E_{\mathrm{sg}}^0=E^0 \setminus E_{\mathrm{fin}}^0$ $E^0=E_{\mathrm{rg}}^0 \cup (E_{\mathrm{fin}}^0)^c$. Fix an open set $U \subset E^*$ and fix a compact set $K \subset E^*$ satisfying $Z(U) \cap Z(K)^c\neq\emptyset$. Write $K=\bigcup_{i=0}^{k}(K\cap E^i)$, for some $k \geq 0$.  Fix $\mu \in Z(U) \cap Z(K)^c$. If $\vert\mu\vert\geq k$ then by Lemma~\ref{Technical Lemma 1} we are done. So we may assume that $\vert\mu\vert <k$. By Lemma~\ref{Technical Lemma 1}, there exists an open neighborhood $V \subset E^{\vert\mu\vert}$ of $\mu$ such that $Z(V) \subset Z(U) \cap Z(\bigcup_{i=0}^{\vert\mu\vert}(K\cap E^i))^c$. Then $s(V)$ is an open neighborhood of $s(\mu)$. Since $\mu \in\partial E$ and $E_{\mathrm{sg}}^0=E^0 \setminus E_{\mathrm{fin}}^0$, we have $s(\mu) \in E^0 \setminus E_{\mathrm{fin}}^0$. Consider the set
\[
F:=\{\alpha_{\vert\mu\vert+1} \in E^1: \alpha \in K, \vert\alpha\vert\geq\vert\mu\vert+1\},
\]
which is a compact subset of $E^1$. Then there exists $e \in r^{-1}(s(V))$ such that $e \notin F$. Take an open neighborhood $W \subset E^1$ of $e$ which does not intersect with $F$. Let
\begin{align*}
O:= \begin{cases}
   (V \times W) \cap E^{\vert\mu\vert+1} &\text{ if $\vert\mu\vert>0$} \\
   W &\text{ if $\vert\mu\vert=0$.}
\end{cases}
\end{align*}
So $\emptyset\neq Z(O) \subset Z(U) \cap Z(K)^c$.
\end{proof}

\begin{defn}[{\cite[Page~202]{MR1478030}}]
Let $\Gamma$ be an \'{e}tale groupoid. Then $\Gamma$ is said to be \emph{essentially free} if the set of elements in $\Gamma^0$ whose isotropy group are trivial form a dense subset of $\Gamma^0$.
\end{defn}

\begin{lemma}\label{equ defn of essentially free}
Let $T$ be a locally compact Hausdorff space and let $\sigma:\dom(\sigma) \to \ran(\sigma)$ be a partial local homeomorphism. Then the Renault-Deaconu groupoid $\Gamma(T,\sigma)$ is essentially free if and only if the set $\{t \in T: t, \sigma(t), \dots \text{ are distinct } \}$ is dense in $T$.
\end{lemma}
\begin{proof}
It is straightforward to see.
\end{proof}

\begin{defn}
Let $E$ be a topological graph. Then $E$ is said to be \emph{essentially free} if the boundary path groupoid $\Gamma(\partial E,\sigma)$ is essentially free.
\end{defn}

The following proposition is a generalization of \cite[Lemma~3.4]{KumjianPaskEtAl:PJM98}.

\begin{prop}\label{top free iff ess free}
Let $E$ be a topological graph. Then $E$ is topologically free if and only if $E$ is essentially free.
\end{prop}
\begin{proof}
First of all, suppose that $E$ is not essentially free. We aim to show that $E$ is not topologically free. Since $E$ is not essentially free, there exists a nonempty open set $N \subset \partial E$ which does not intersect with $\{\mu \in \partial E: \mu, \sigma(\mu), \dots \text{ are distinct } \}$ due to Lemma~\ref{equ defn of essentially free}. For each $0 \leq p<q$, define a closed subset of $\partial E$ to be $B_{p,q}:=\{\mu \in \partial E: \sigma^p(\mu)=\sigma^q(\nu)\}$. For each $0 \leq p<q$, let $A_{p,q}:=B_{p,q} \cap N$, then $A_{p,q}$ is a closed subset of the subspace $N$. Notice that $N=\bigcup_{0 \leq p<q}A_{p,q}$ because $N$ does not intersect with $\{\mu \in \partial E: \mu, \sigma(\mu), \dots \text{ are distinct } \}$. By the Baire's category theorem, there exists a nonempty open subset $O$ of $N$ ($O$ is also open in $\partial E$) contained in $A_{p_0,q_0}$ for some $1 \leq p_0 <q_0$. By Definition~\ref{def of boundary path of top graph}, there exist an open set $U \subset E^*$ and a compact set $K \subset E^*$ such that $\emptyset \neq Z(U) \cap Z(K)^c \subset O$. By Lemma~\ref{Technical Lemma 1}, there exist $n \geq 0$ and an open subset $V \subset E^n$ such that $\emptyset \neq Z(V) \subset O$. We may assume that $n \neq 0$. Let $W:=(V \times (r^p)^{-1}(s(V))) \cap E^{n+p_0}$. Then $Z(W) \subset Z(V)$. So $Z(W) \subset A_{p_0,q_0}$. We deduce that every $\mu \in Z(W)$ satisfies that $\sigma^{n+p_0}(\mu)=\sigma^{n+q_0}(\mu)$ since $\mu \in A_{p_0,q_0}$. \cite[Proposition~2.8]{Katsura:TAMS04} assures that for each $\alpha \in W$ there exists $\mu \in Z(W)$ such that $\mu=\alpha\beta$. We then conclude that $s(W)$ is an open subset of $E^0$ consisting of base points of cycles. We claim that $s(W)$ consists of base points of cycles without entrances. Suppose not, for a contradiction. We obtain two infinite paths $\alpha\beta,\alpha\beta' \in Z(W)$, where $\alpha \in W, \beta_1\neq \beta_2$. Since $\nu:=\beta_1\dots\beta_{q_0-p_0},\nu':=\beta_1'\dots\beta'_{q_0-p_0}$ are cycles, $\alpha\nu\nu'\nu\nu'\dots \in Z(W) \notin A_{p_0,q_0}$, which is a contradiction. Hence $s(W)$ is an open subset of $E^0$ consisting of base points of cycles. Therefore $E$ is not topologically free.

Conversely suppose that $E$ is essentially free. Suppose that $E$ is not topologically free, for a contradiction. By Lemma~\ref{equ defn of essentially free}, the set $\{\mu \in \partial E: \mu, \sigma(\mu), \dots \text{ are distinct } \}$ is dense in $\partial E$. By \cite[Proposition~6.12]{Katsura:ETDS06}, there exist a nonempty open set $V \subset E^0$ consisting of base points of cycles without entrances, and a homeomorphism $\sigma$ on $V$ such that $\sigma=r \circ (s\vert_{r^{-1}(V)})^{-1}$. By \cite[Proposition~2.8]{Katsura:TAMS04}, $V \subset E_{\mathrm{rg}}^0$. Fix a vertex $v \in V$. Let $\nu$ be the unique simple cycle such that $r(\nu)=v$. By the assumption, there is a convergent sequence $(\nu^{(n)})_{n=1}^{\infty} \subset \{\mu \in \partial E:  \mu, \sigma(\mu), \dots \text{ are distinct }\}$ with the limit $\nu$. By \cite[Lemma~4.8]{KumjianLiTwisted2}, $r(\nu^{(n)}) \to r(\nu)=v$, and so there exists $N \geq 1$ such that $r(\nu^{(N)}) \in V$. Let $\alpha$ be the unique simple cycle with $r(\alpha)=r(\nu^{(N)})$.  Since $V$ consists of base points of cycles without entrances and $V \subset E_{\mathrm{rg}}^0$, we deduce that $\nu^{(N)}=\alpha\alpha\cdots$, which is a contradiction. So $E$ is topologically free.
\end{proof}

\begin{defn}[{\cite[Definition~2.1]{MR1478030}}]\label{define loc contracting}
Let $\Gamma$ be an \'{e}tale groupoid. Then $\Gamma$ is said to be \emph{locally contracting} if for any nonempty open set $U \subset \Gamma^0$, there exist an open subset $V \subset U$ and an open bisection $N \subset \Gamma$ such that $\overline{V} \subset s(N)$ and $s \circ r \vert_{N^{-1}}^{-1}(\overline{V}) \subsetneqq V$.
\end{defn}

\begin{defn}\label{define loc contracting top graph}
Let $E$ be a topological graph. Then $E$ is said to be \emph{locally contracting} if the boundary path groupoid $\Gamma(\partial E,\sigma)$ is locally contracting.
\end{defn}

\begin{defn}
Let $E$ be a topological graph and let $v \in E^0$. Then $v$ is said to \emph{connect to a cycle} if there exists $u \in E^0$ such that $v \in \mathrm{Orb}^+(u)$ and $u$ is the base point of a cycle.
\end{defn}

The following theorem is a generalization of partial results from \cite[Theorem~3.9]{KumjianPaskEtAl:PJM98}.

\begin{thm}\label{sufficient condition 1 for boundary path groupoid loc con}
Let $E$ be a topologically free totally disconnected topological graph. Suppose that the following subset of $E^0$ is dense in $E^0$.
\begin{align*}
B:=\{v : &\text{ $v$ connects to a cycle $\mu$ with entrances such that any open neighborhood $N$ }
\\&\text{of $\mu$ contains an open neighborhood $U$ of $\mu$ with $r^{\vert\mu\vert}(U) \subset s^{\vert\mu\vert}(U)$} \}.
\end{align*}
Then $E$ is essentially free and locally contracting. Hence $\mathcal{O}(E)$ is purely infinite.
\end{thm}
\begin{proof}
Since every cycle of $E$ has entrances, $E$ is topologically free. So $E$ is essentially free by Proposition~\ref{top free iff ess free}.

We claim that $E_{\mathrm{sce}}^0=\emptyset$. Suppose that $E_{\mathrm{sce}}^0\neq\emptyset$, for a contradiction. Fix $v \in E_{\mathrm{sce}}^0$. Then there exists an open neighborhood $V$ of $v$ not intersecting with $\overline{r(E^1)}$. Since $B$ is dense in $E^0,V \cap B \neq \emptyset$ which is a contradiction. So $E_{\mathrm{sce}}^0=\emptyset$ and we finish proving the claim.

Now we prove that $E$ is locally contracting. Fix a nonempty open set $N \subset \partial E$. By Lemma~\ref{Technical Lemma 2} there exists a nonempty open subset $U$ of $E^*$ such that $\emptyset\neq Z(U) \subset N$. Then $s(U)$ is a nonempty open subset of $E^0$. Since $B$ is dense in $E^0$, there exist $\mu\nu e \beta, \mu\nu \alpha e\beta \in \partial E$ such that $\mu \in U, \alpha$ is a cycle, and $e \neq \alpha_1$. We may assume that $\vert \mu \nu \vert>0$ (the case $\vert \mu \nu \vert=0$ would follow a similar argument). Take a compact open $s^{\vert\mu\nu\vert}$-sections $O_1$ such that $\mu\nu \in O_1$ and $Z(O_1) \subset Z(U)$; and take compact open $s^{\vert\alpha\vert}$-sections $O_2$ such that $\alpha \in O_2, r^{\vert\alpha\vert}(O_2) \subset s^{\vert\mu\nu\vert}(O_1) \cap s^{\vert\alpha\vert}(O_2), e \notin \{\alpha_1':\alpha' \in O_2\}$. Define a compact open bisection of the boundary path groupoid $\Gamma(\partial E,\sigma)$ by $S:=\mathcal{U}(Z((O_1 \times O_2)\cap E^{\vert\mu\nu\alpha\vert}),Z(O_1),\vert\mu\nu\alpha\vert,\vert\mu\nu\vert)$. Define a nonempty compact open subset of $N$ by $W:=s(S)$. Then $\overline{W} \subset s(S)$. For any $\mu'\nu',\mu''\nu'' \in O_1, \alpha' \in O_2,\beta \in \partial E$ such that $(\mu'\nu'\alpha'\beta,\vert\mu\nu\alpha\vert-\vert\mu\nu\vert,\mu''\nu''\beta) \in S$, there exist unique $\mu''',\nu''' \in O_1, \alpha''' \in O_2$ such that $\mu'''\nu'''\alpha'''\alpha'\beta \in Z(Z((O_1 \times O_2)\cap E^{\vert\mu\nu\alpha\vert}))$. So $\mu'\nu'\alpha'\beta=s(\mu'''\nu'''\alpha'''\alpha'\beta,\vert\mu\nu\alpha\vert-\vert\mu\nu\vert,\mu'\nu'\alpha'\beta)\in W$. Hence $s \circ r_{S^{-1}}^{-1}(\overline{W}) \subset W$. Pick up an arbitrary $\gamma \in s(e)\partial E$ (see \cite[Proposition~2.8]{Katsura:TAMS04}). Then $\mu\nu e \gamma \in W$ but $\mu\nu e \gamma \notin s \circ r_{S^{-1}}^{-1}(\overline{W})$ because $e \notin \{\alpha_1':\alpha' \in O_2\}$. Therefore $s \circ r_{S^{-1}}^{-1}(\overline{W}) \subsetneqq W$. By Definition~\ref{define loc contracting top graph}, $\Gamma(\partial E,\sigma)$ is locally contracting.

By Theorem~\ref{groupid model for twisted top graph alg}, $\mathcal{O}(E)$ is isomorphic with $C^*(\Gamma(\partial E,\sigma))$. Since $E$ is essentially free and locally contracting, \cite[Proposition~2.4]{MR1478030} gives $\mathcal{O}(E)$ is purely infinite.
\end{proof}

\begin{cor}\label{totally disconn PI top graph alg}
Let $E$ be a totally disconnected topological graph such that $\mathcal{O}(E)$ is simple. Suppose that there exists a cycle $\mu$ with entrances such that any open neighborhood $N$ of $\mu$ contains an open neighborhood $U$ of $\mu$ with $r^{\vert\mu\vert}(U) \subset s^{\vert\mu\vert}(U)$. Then $E$ is essentially free and locally contracting. Hence $\mathcal{O}(E)$ is purely infinite.
\end{cor}
\begin{proof}
Since $\mathcal{O}(E)$ is simple, by Theorem~\ref{simplicity} $E$ is topologically free. By Proposition~\ref{top free iff ess free} $E$ is essentially free. Fix $v \in E^0$ and fix an open neighborhood $V$ of $v$. Since $\mathcal{O}(E)$ is simple, by Theorem~\ref{simplicity} $E$ is cofinal. So $\bigcup_{i=1}^{\vert \mu \vert} \mathrm{Orb}^+(s(\mu_i))$ is dense in $E^0$. Hence $V \cap  \Big(\bigcup_{i=1}^{N} \mathrm{Orb}^+(s(\mu_i)) \Big) \neq \emptyset$. Thus $V \cap B\neq\emptyset$ (see Theorem~\ref{sufficient condition 1 for boundary path groupoid loc con}). Theorem~\ref{sufficient condition 1 for boundary path groupoid loc con} implies that $E$ is locally contracting. Therefore \cite[Proposition~2.4]{MR1478030} yields that $\mathcal{O}(E)$ is purely infinite.
\end{proof}

\section{Concluding Remarks}

Katsura in \cite{Katsura:JFA08} defined a concept called contracting topological graphs which can provide purely infinite topological graph algebras. We recall the definition of contracting topological graphs and state Katsura's result.

\begin{defn}[{\cite[Definition~2.3]{Katsura:JFA08}}]\label{define contracting top graph}
Let $E$ be a topological graph. A nonempty precompact open set $V \subset E^0$ is said to be \emph{contracting} if there exists a finite family of nonempty open sets $\{U_i \subset E^{n_i}:n_i \geq 1\}_{i=1}^{k}$ satisfying the following.
\begin{enumerate}
\item\label{condition 1} $r(U_i) \subset V, i=1,\dots,k$;
\item for $i \neq j, n_i \leq n_j$, we have $\{(\mu_1,\dots,\mu_{n_i}): (\mu_1,\dots,\mu_{n_i}) \in U_i,\mu \in E^{n_j}\}=\emptyset$; and
\item\label{condition 3} $\overline{V} \subsetneqq \bigcup_{i=1}^{k}s(U_i)$.
\end{enumerate}
Moreover, $E$ is said to be \emph{contracting} at a vertex $v \in E^0$ if $\overline{\mathrm{Orb}^+(v)}=E^0$ and every open neighborhood of $v$ contains a contracting open set. Furthermore, $E$ is said to be \emph{contracting} if $E$ is contracting at some vertex in $E^0$.
\end{defn}

\begin{thm}[{\cite[Theorem~A]{Katsura:JFA08}}]\label{contracting minimal gives PI}
Let $E$ be a topological graph. Suppose that $\mathcal{O}(E)$ is simple. If $E$ is contracting then $\mathcal{O}(E)$ is purely infinite.
\end{thm}

We provide two examples which indicate that both of pure infiniteness conditions for simple totally disconnected topological graph algebras in Corollary~\ref{totally disconn PI top graph alg} and Theorem~\ref{contracting minimal gives PI} are not comparable.

\begin{eg}
Define $E^0:=\{v\}; E^1:=\{e_0,e_1\}$. Then $\mathcal{O}(E) \cong \mathcal{O}_2$ which is simple and purely infinite. Notice that $e_0$ is a cycle with entrances and $\{e_0\}$ is an open neighborhood of $e_0$ with $r(\{e_0\}) \subset s(\{e_0\})$. So the assumption of Corollary~\ref{totally disconn PI top graph alg} is satisfied. However, it is easy to see that this graph is not contracting. Therefore the assumption of Theorem~\ref{contracting minimal gives PI} is not satisfied. Furthermore, Katsura in \cite[Remark~2.8]{Katsura:JFA08} asked that whether the converse of Theorem~\ref{contracting minimal gives PI} is true and this is a counterexample of the question.
\end{eg}

\begin{eg}
Define $E^0:=\{v\}; E^1:=\{e_0,e_1\}$. This time we consider the dual graph of $\widehat{E}:=(\prod_{n=1}^{\infty}\{0,1\},\prod_{n=1}^{\infty}\{0,1\},\id,\sigma)$ where $\sigma$ is the one-sided shift. By \cite[Theorem~7.5]{KumjianLiTwisted2}, $\mathcal{O}(\widehat{E}) \cong \mathcal{O}(E)$ so $\mathcal{O}(\widehat{E})$ is simple and purely infinite. Since every cycle of $\widehat{E}$ has no entrances, the assumption of Corollary~\ref{totally disconn PI top graph alg} is not satisfied. On the other hand, $\overline{\mathrm{Orb}^+(000 \dots)}=\widehat{E}^0$ and $Z(\underbrace{0 \cdots 0}_{n}) \subset \widehat{E}^0$ is contracting for all $n \geq 1$. To see that $Z(\underbrace{0 \cdots 0}_{n})$ is contracting, let $U_1:=Z(\underbrace{0 \cdots 0}_{n+1}) \subset \widehat{E}^1, U_2:=Z(\underbrace{0 \cdots 0}_{n}1)\subset \widehat{E}^1$. It is straightforward to check that $\{U_1,U_2\}$ satisfies Conditions~(\ref{condition 1})--(\ref{condition 3}) of Definition~\ref{define contracting top graph}. So $Z(\underbrace{0 \cdots 0}_{n})$ is contracting. Hence $E$ is contracting and the assumption of Theorem~\ref{contracting minimal gives PI} is satisfied.
\end{eg}

Yeend in \cite{Yeend:CM06} showed that topological $1$-graph $C^*$-algebras coincide with topological graph algebras (see also \cite{KumjianLiTwisted2}). Later, Renault, Sims, Williams, and Yeend in \cite{RSWY} provided a sufficient condition for simple compactly-aligned topological higher-rank graph $C^*$-algebras to be purely infinite. In the following we interpret their condition in the topological graph setting and present their result.

\begin{defn}[{\cite[Definition~5.7]{RSWY}}]\label{define contracting top 1-graph}
Let $E$ be a topological graph. A nonempty precompact open set $U \subset E^0$ is said to be \emph{contracting} if there exist $0 \leq n<m$, a nonempty precompact open $s^n$-section $Y_n$, and a nonempty precompact open $s^m$-section $Y_m$, such that
\begin{enumerate}
\item $s^m(Y_m)=s^n(Y_n)$;
\item $\overline{r(Y_m)} \subset r(Y_n)=U$;
\item for $\mu \in Y_m,\nu \in Y_n$ with $r^m(\mu)=r^n(\nu)$, there exists $\gamma \in E^{m-n}$ such that $\mu=\nu\gamma$;
\item there exists a nonempty open subset $W$ of $Y_n E^*$ such that $\{\nu_1\dots\nu_n:\nu \in W\}=Y_n$ and that for $\mu \in Y_m,\nu \in W$ there is no $\gamma \in E^*$ satisfying $\mu=\nu\gamma$ or $\nu=\mu\gamma$.
\end{enumerate}
\end{defn}

\begin{thm}[{\cite[Proposition~5.8]{RSWY}}]\label{purely inf top 1-graph}
Let $E$ be a topological graph. Suppose that $\mathcal{O}(E)$ is simple and that for any $v \in E^0$, there exist $n \geq 0$ and an open subset $U$ of $E^n$ satisfying that $v \in r(U)$ and $s(U)$ is contracting in the sense of Definition~\ref{define contracting top 1-graph}. Then $\mathcal{O}(E)$ is purely infinite.
\end{thm}



The disadvantage of this theorem is that in order to prove the pure infiniteness of a topological graph algebra, one has to check the contracting condition for every vertex.

Overall, both criteria of Katsura and Renault-Sims-Williams-Yeend for topological graph algebras being purely infinite relies on the assumption of the given topological graph algebras being simple. Our approach does not but pay the price that we have to assume that the starting topological graphs are totally disconnected.

Finally, we consider a compact totally disconnected topological graph $E$ such that the range, source map are surjective and $\mathcal{O}(E)$ is simple. Schafhauser recently proved that $\mathcal{O}(E)$ is finite if and only if the source map is injective (see \cite[Theorem~6.7]{MR3354440}). On the other hand, it is hard to get a completely graphic characterization for $\mathcal{O}(E)$ to be purely infinite. However, by combining the recent work of Brown, Clark, and Sierakowski in \cite{MR3456600} and the work of Kumjian and Li \cite{KumjianLiTwisted2}, we are able to reduce the problem by only considering the projections on the vertex space. More precisely, $\mathcal{O}(E)$ is purely infinite if and only if every nonzero projection on $C(E^0)$ is infinite.


\begin{thebibliography}{15}
\bibitem{MR1478030} C. Anantharaman-Delaroche, \emph{Purely infinite {$C^*$}-algebras arising from dynamical systems}, Bull. Soc. Math. France \textbf{125} (1997), 199--225.
\bibitem{MR3456600} J. Brown, L.O. Clark, and A. Sierakowski, \emph{Purely infinite {$C^\ast$}-algebras associated to \'etale groupoids}, Ergodic Theory Dynam. Systems \textbf{35} (2015), 2397--2411.
\bibitem{BCW} N. Brownlowe, T.M. Carlsen, and M.F. Whittaker, \emph{Graph algebras and orbit equivalence}, preprint, arXiv:1410.2308.
\bibitem{Katsura:TAMS04} T. Katsura, \emph{A class of {$C^\ast$}-algebras generalizing both graph
    algebras and homeomorphism {$C^\ast$}-algebras {I}. {F}undamental results}, Trans. Amer. Math.
    Soc. \textbf{356} (2004), 4287--4322.
\bibitem{Katsura:ETDS06} T. Katsura, \emph{A class of {$C\sp *$}-algebras generalizing both graph algebras and homeomorphism {$C\sp *$}-algebras {III}. {I}deal structures}, Ergodic Theory Dynam. Systems \textbf{26} (2006), 1805--1854.
\bibitem{Katsura:JFA08} T. Katsura, \emph{A class of {$C\sp *$}-algebras generalizing both graph algebras and homeomorphism {$C\sp *$}-algebras. {IV}. {P}ure infiniteness}, J. Funct. Anal. \textbf{254} (2008), 1161--1187.
\bibitem{MR2563503} T. Katsura, \emph{Cuntz-{K}rieger algebras and {$C^\ast$}-algebras of topological graphs}, Acta Appl. Math. \textbf{108} (2009), 617--624.
\bibitem{Katsura:JFA04} T. Katsura, \emph{On {$C\sp *$}-algebras associated with {$C\sp *$}-correspondences}, J. Funct. Anal. \textbf{217} (2004), 366--401.
\bibitem{Kirchbergclassification} E. Kirchberg, \emph{The classification of purely infinite $C^*$-algebras using Kasparov's theory}, preprint, 1994.
\bibitem{KumjianLiTwisted2} A. Kumjian and H. Li, \emph{Twisted topological graph algebras are twisted groupoid $C^*$-algebras}, J. Operator Theory, to appear, arXiv:1507.04449.
\bibitem{KumjianPaskEtAl:PJM98} A. Kumjian, D. Pask, and I. Raeburn, \emph{Cuntz-{K}rieger algebras of directed graphs}, Pacific J. Math. \textbf{184} (1998), 161--174.
\bibitem{KumjianPaskEtAl:JFA97} A. Kumjian, D. Pask, I. Raeburn, and J. Renault, \emph{Graphs, groupoids, and {C}untz-{K}rieger algebras}, J. Funct. Anal. \textbf{144} (1997), 505--541.
\bibitem{Lance:Hilbert$C*$-modules95} E.C. Lance, Hilbert {$C\sp *$}-modules, A toolkit for operator algebraists, Cambridge University Press, Cambridge, 1995, x+130.
\bibitem{Phillips:DM00} N.C. Phillips, \emph{A classification theorem for nuclear purely infinite simple {$C\sp *$}-algebras}, Doc. Math. \textbf{5} (2000), 49--114.
\bibitem{Pimsner:FIC97} M.V. Pimsner, \emph{A class of {$C^*$}-algebras generalizing both
    {C}untz-{K}rieger algebras and crossed products by {${\bf Z}$}}, Fields Inst. Commun., 12, Free
    probability theory ({W}aterloo, {ON}, 1995), 189--212, Amer. Math. Soc., Providence, RI, 1997.
\bibitem{RaeburnWilliams:Moritaequivalenceand98} I. Raeburn and D.P. Williams, Morita equivalence and continuous-trace {$C\sp *$}-algebras, American Mathematical Society, Providence, RI, 1998, xiv+327.
\bibitem{Renault:groupoidapproachto80} J. Renault, A groupoid approach to {$C\sp{\ast} $}-algebras, Springer, Berlin, 1980, ii+160.
\bibitem{Renault:00}J. Renault, \emph{Cuntz-like algebras}, Operator theoretical methods (Timi\c soara, 1998), 371--386, Theta Found., Bucharest, 2000.
\bibitem{RSWY} J. Renault, A. Sims, D.P. Williams, and T.Yeend, \emph{Uniqueness theorems for topological higher-rank graph $C^*$-algebras}, preprint, arXiv:0906.0829.
\bibitem{RordamStormer:Classificationofnuclear02} M. R{\o}rdam and E. St{\o}rmer, Classification of nuclear {$C\sp *$}-algebras {E}ntropy in operator algebras, Operator Algebras and Non-commutative Geometry, 7, Springer-Verlag, Berlin, 2002, x+198.
\bibitem{MR3354440} C.P. Schafhauser, \emph{Finiteness properties of certain topological graph algebras}, Bull. Lond. Math. Soc. \textbf{47} (2015), 443--454.
\bibitem{Yeend:CM06}T. Yeend,\emph{Topological higher-rank graphs and the {$C\sp *$}-algebras of topological 1-graphs}, Contemp. Math., 414, Operator theory, operator algebras, and applications, 231--244, Amer. Math. Soc., Providence, RI, 2006.
\end{thebibliography}
\end{document}